\documentclass[a4paper]{article}

\usepackage[english]{babel}
\usepackage[utf8]{inputenc}
\usepackage{amsmath}
\usepackage{graphicx}
\usepackage[colorinlistoftodos]{todonotes}

\usepackage{amssymb} 
\usepackage{epstopdf} 
\usepackage{amsmath} 
\usepackage{amsthm} 
\usepackage{mathrsfs} 
 
\newtheorem{theorem}{Theorem}[section] 
\newtheorem*{m_t_1}{Theorem 4.1}
\newtheorem*{m_t_2}{Theorem 5.1}
\newtheorem*{m_t_3}{Theorem 5.3}
\newtheorem*{theorem-A}{Theorem A}
\newtheorem{lemma}[theorem]{Lemma} 
 
\newtheorem{proposition}[theorem]{Proposition} 
\newtheorem{definition}[theorem]{Definition} 
 
\newtheorem{remark}[theorem]{Remark}

\DeclareMathOperator{\rk}{rk} 
\DeclareMathOperator{\gk}{\mathscr{GK}-dim} 
\DeclareMathOperator{\s}{\ast} 
\DeclareMathOperator{\Hom}{Hom}  
\DeclareMathOperator{\supp}{Supp}  
 
\DeclareMathOperator{\Kdim}{K.dim} 

\title{Representations of the $n$-dimensional quantum torus}

\author{Ashish Gupta}

\date{}

\begin{document}
\maketitle

\begin{abstract}
The $n$-dimensional quantum torus $\mathcal O_{\mathbf q}((F^\times)^n)$ is defined as the associative $F$-algebra generated by $x_1, \cdots, x_n$ together with their inverses satisfying the relations $x_ix_j = q_{ij}x_jx_i$, where $\mathbf q = (q_{ij})$. We show that the modules that are finitely generated over certain  commutative sub-algebras $\mathscr B$ are $\mathscr B$-torsion-free and have finite length. 
We determine the Gelfand-Kirillov dimensions of simple modules in the case when \[ \Kdim(\mathcal O_{\mathbf q}((F^\times)^n)) = n - 1, \] where $\Kdim$ stands for the Krull dimension.  In this case if $M$ is a simple $\mathcal O_{\mathbf q}((F^\times)^n)$-module then $ \gk(M) = 1$ or \[ \gk(M) \ge \gk(\mathcal O_{\mathbf q}((F^\times)^n)) - \gk(\mathcal Z(\mathcal O_{\mathbf q}((F^\times)^n)))  - 1,\]
where $\mathcal Z(C)$ stands for the center of an algebra $C$. We also show that there always exists a simple $F \s A$-module satisfying the above inequality.
\end{abstract}

\section{Introduction}

The $n$-dimensional quantum torus $\mathcal O_{\mathbf q}((F^\times)^n)$ is defined as the (associative) $F$-algebra  
which is generated by the variables $x_1, x_2, \cdots, x_n$ together with their inverses satisfying the relations 
\begin{equation}
\label{def_reln}
x_ix_j = q_{ij}x_jx_i
\end{equation}
where $1 \le i,j \le n$ and $\mathbf q  = (q_{ij})$.
  
It plays an important role in non-commutative geometry and the theory of quantum groups.
The case $n = 2$ happens  to be relatively well studied (see \cite{Lo} and \cite{Ja}).  Here we consider the general case and in particular the structure and growth of modules.

It is well-known that the $n$-dimensional quantum torus has the structure of a twisted group algebra $F \s A$ of a free abelian group of rank $n$ over the field $F$.   The subgroups 
$B \le A$ so that the corresponding sub-algebra $F \s B$ commutative play an important role. For example we have the following result.

\begin{theorem}[Theorem A of \cite{Brookes:2000}]
\label{Brookes_thm}
The Krull and global dimensions of a twisted group algebra $F \s A$ equals the supremum of the ranks of subgroups $B \le A$  so that $F \s B$ is commutative. 
\end{theorem}

Our first result describes the structure of the $F \s A$-modules that are finitely generated over a sub-algebra $F \s B$ such that $F \s B$ is commutative.

\begin{m_t_1}
Suppose that the quantum torus $F \s A$ has center $F$. Let $M$ be a nonzero finitely  
generated  $F \s A$-module.  
Let $C < A$ be a subgroup having a subgroup $C_0$ of  
finite index such that $F \s C_0$ is commutative.  
If $M$ is finitely generated as an $F \s C$-module then,    
\begin{enumerate} 
\item[(i)] $\gk(M) = \rk(C)$,  
\item[(ii)] $M$ is $F \s C$-torsion-free,  
\item[(iii)] $M$ has finite length,  
\item[(iv)] $M$ is cyclic.      
\end{enumerate}  
\end{m_t_1}  

\subsection{Gelfand-Kirillov dimension of simple modules.}
Let  $\mathcal A$ be an affine algebra of finite Gelfand--Kirillov dimension (GK dimension) say $d$. If $M$ is a finitely generated $\mathcal A$-module then \[ 
0 \le \gk(M)  \le \gk(\mathcal A) \]holds (e.g., \cite[Proposition 5.1(d)]{KL:2000}).
However, not all values between $0$ and $\gk(\mathcal A)$ can be attained by the GK dimension of a finitely generated $\mathcal A$-module. For example, as is well known if $\mathcal A = A_n(k)$ the  $n$-th Weyl algebra over a field of characteristic zero, then the famous inequality of Bernstein says that  \[ \gk(M)  \ge \frac{1}{2}\gk(A_n(k)). \] 

The question arises as to what values can be assumed by the GK dimensions of simple $\mathcal A$-modules. For example, if $n \le l \le 2n - 1$ there exists a simple $A_n(\mathbb C)$-module $S_l$ with $\gk(S_l) = l$ (see \cite{Coutinho:1995}). The next question we take up concerns the GK  dimension of simple modules over the quantum tori. In \cite{MP}, it was shown that if $\dim (F \s A ) = 1$ then each simple 
$F \s A$-module $N$ satisfies $\gk(N) =  \rk(A) - 1$. 
Here $\dim(F \s A)$ stands for either the Krull or the global dimension of the quantum torus $F \s A$.

Note that $\dim(F \s A)$ satisfies 
\[ 1 \le \dim(F \s A) \le \rk(A). \]
If $\dim(F \s A) = \rk(A)$ then by Theorem \ref{Brookes_thm}
the algebra $F \s A$ is a finite normalizing extension of $F \s A'$, where the sub-algebra $F \s A'$ is commutative and so (e.g., \cite[Exercise 15A.3]{Rowen:2008}) 
a simple $F \s A$-module is a finite direct sum of simple $F \s A'$-modules.  Using this and the fact that simple modules over commutative affine algebras are finite dimensional it can be easily deduced that in this case $\gk(S) = 0$  for each simple $F \s A$-module $S$. 

Here we consider the case $\dim(F \s A) = n - 1$. We obtain the following results.

\begin{m_t_2}
Let $M$ be a simple-$F \s A$-module, where $\dim(F \s A) =  \rk(A) - 1$. Let $Z$ be the subgroup of $A$ such that $\mathcal Z(F \s A) = F \s  Z$. Then either \[ \gk(M) = 1, or \] 
\[ \gk(M) \ge \rk(A) - \rk(Z) - 1.  \]    
\end{m_t_2}

There always exists a simple $F \s A$-module satisfying the inequality given in the last theorem.

\begin{m_t_3}
Suppose that $\dim(F \s A) = \rk(A) - 1$. There exists a simple $F \s A$-module $M$ with \[ \gk(M) \ge \rk(A) - \rk(Z) - 1, \] where $Z$ is the unique subgroup of $A$ so that 
\[ \mathcal Z(F \s A) = F \s Z. \] 

\end{m_t_3}

Finally, we remark that it seems to be quite difficult to determine the possible values of GK-dimensions of simple modules over the quantum tori $F \s A$ such that 
\[ 2 \le \dim(F \s A) \le n - 2 \]It seems to depend on the defining parameters $q_{ij}$

\section{The twisted group algebra structure}

\label{TGA}

Let $F^* := F \setminus \{0\}$. Let $A$ denote a finitely generated free abelian group. 
We denote by $\rk(A)$ the (torsion-free ) rank of $A$. 
An $F$-algebra $\mathcal A$ is a \emph{twisted group algebra} $F \s A$ of $A$ over $F$ 
if $\mathcal A$ has a copy $\overline A := \{ \bar a : a \in A \}$ of $A$ which is an $F$-basis and such that the multiplication in 
$\mathcal A$ satisfies  
\begin{equation}
\label{tws}
\bar {a}_1 \bar {a}_2 = \tau(a_1, a_2)\overline{a_1a_2} \ \ \ \ \ \forall a_1, a_2 \in A,
\end{equation} 
where $\tau : A \times A \rightarrow F^*$ is a function satisfying
\[ \tau(a_1, a_2)\tau(a_1a_2, a_3) = \tau(a_2, a_3)\tau(a_1, a_2a_3) \ \ \ \ \ \forall a_1, a_2, a_3 \in A.  \]
For $a_1, a_2 \in A$, it easily follows from (\ref{tws}) that the group-theoretic commutator $[\bar a_1, \bar a_2] \in F^*$. 
\subsection{Commutator calculus}
\label{com_cal}
It follows from the
basic properties of commutators (e.g., \cite[Section 5.1.5]{Rob:1996}) that 
\begin{align}\label{com_id_1}
[\bar {a}_1 \bar {a}_2, \bar {a}_3] &= [\bar {a}_1, \bar {a}_3] [\bar {a}_2, \bar {a}_3],\\ \label{com_id_2}
[\bar {a}_1,  \bar {a}_2 \bar {a}_3] &= [\bar {a}_1, \bar {a}_2] [\bar {a}_1, \bar {a}_3], \\ \label{com_id_3}
[\bar {a}_1, \bar {a}_2^{-1}] &= [\bar {a}_1, \bar {a}_2]^{-1}, \\ \label{com_id_4} 
[\bar {a}_1^{-1}, \bar {a}_2] &= [\bar {a}_1, \bar {a}_2]^{-1}\ \ \ \ \ \forall a_1, a_2, a_3 \in A. 
\end{align}
For a subset $X$ of $A$, we define $\overline{X} = \{\bar x : x \in X \}$. 
If $X, Y \subset A$, we set 
\[ [\overline X, \overline Y] = \langle [\bar x, \bar y] : x \in X, y \in Y \rangle. \]
If $\alpha \in F \s A$, we may express $\alpha = \sum_{a \in A} \lambda_a \bar a$, where $\lambda_a \in F$ and  $\lambda_a = 0$ for almost all $a \in A$.
We define the support of $\alpha$ (in $A$) as  
\[ \supp(\alpha) = \{ a \in A \mid \lambda_a \ne 0 \}.\]
Note that for a subgroup $B$ of $A$, the sub-algebra generated by $\overline B \subset F \s A$ is a twisted group algebra $F \s B$. It is described as 
\[ F \s B = \{ \beta \in F \s A \mid \supp(\beta) \subset B\} \]

\subsection{The center}
It was shown in \cite[Proposition 1.3]{MP} that an algebra $F \s A$ is simple if and only if it has center $F$. The following fact might be quite expected.

\begin{proposition}
\label{c_a_fis}
An algebra $F \s A$ has center exactly $F$ if and only if for each subgroup $A_1 < A$ with finite index 
$F \s A_1$ has center $F$.
\end{proposition}
\begin{proof}
Suppose that $F \s A$ has center $F$. Let $A_1 \le A$ be a subgroup such that $l : = [A : A_1] < \infty$. 
We claim that $F \s A_1$ also has center $F$. Using \cite[Proposition 1.3]{MP}, we may assume that $\bar {a}_1$ is central in $F \s A_1$ for $1 \ne a_1 \in A_1$. For any $a \in A$, (\ref{com_id_1}) and (\ref{com_id_2}) yield:
\[ [ \bar {a}_1^l, \bar a ] = [ \bar {a}_1, \bar a ]^l =  [ \bar {a}_1, \bar a^l ] = 1, \] 
where the last equality holds since $a^l \in A_1$. 
Since $A$ is torsion-free by definition, $1 \ne a_1^l$. Thus $\bar {a}_1^l$ is a non-scalar central element of $F \s A$. 
The converse is clear. 
\end{proof} 

It may also be expected that the center of a twisted group algebra has the form $F \s B$ for a subgroup $B \le A$.

\begin{proposition}
The center of a twisted group algebra $F \s A$ has the form $F \s B$ for a suitable subgroup $B$ in $A$.
\end{proposition}

\begin{proof}
Let $\mathcal Z$ be he center of the algebra $F \s A$ and
$\zeta \in \mathcal Z$. Write
$\zeta = \sum_{i = 1}^t \lambda_iz_i$.
Let $A = \langle x_1, \cdots, x_n \rangle$
The condition $\zeta \in \mathcal Z$ is equivalent to the $n$ conditions $[\bar {x}_i, \zeta] = 0$, where $1 \le i \le n$ and $[a, b]$ stands for the lie commutator $[a, b] = ab - ba$.
Now 
\begin{align*}
[\bar x_i, \zeta ] &= [\bar x_i, \sum_{j = 1}^t \lambda_j a_j ] \\
&= \sum_{j = 1}^t \lambda_j[x_i, a_j ].
\end{align*}

From the defining relations we have $[\bar x, \bar y] = \gamma(x,y) \overline {xy}$, where $\gamma(x,y) \in F$.  
Therefore 
\[ [\bar x_i, \zeta ] = \sum_{j = 1}^t \lambda_j\gamma(\bar {x}_i,\bar {a}_j) \overline {x_ia_j}. \] 
Hence $[x_i, \zeta] = 0$ if and only if $\gamma(\bar{x}_i, \bar {a}_j) = 0$. But this means that $[\bar{x}_i, \bar {a}_j] = 0$ and so $\bar a_j \in \mathcal Z$. Thus 
$\mathcal Z = F \s B$ where 
\[ B : = \langle \cup \supp(\zeta) \mid \zeta \in \mathcal Z \rangle. \]  
\end{proof}

\subsection{Localization}
It is well-known (see, for example, \cite[Lemma 37.8]{Pass:1989}) that for subgroups $B \le A$ the subsets $\mathcal X_B: =  F \s B \setminus \{0\}$ are Ore subsets in the algebra $F \s A$ and the latter has a (right) Ore localization with respect to $\mathcal X_B$ which we denote as $(F \s A)\mathcal X_B^{-1}$.

The localization has the structure of a crossed product of the group $A/B$ over the quotient division ring $\mathscr D_B$ of $F \s B$. We refer the reader to \cite{Pass:1989} for details on crossed products. Recall that if $M$ is an $F \s A$-module the corresponding localization $M(\mathcal X)^{-1}$ is a module for $(F \s A)\mathcal X_B^{-1}$. We refer the reader to \cite{GW:1989} for a discussion of localization in a general non-commutative setting.   

Let $M$ be an $F \s A$-module and $C \le A$. As we noted above $\mathcal X : = F \s C \setminus \{0\}$ is an Ore subset in $F \s A$. The subset
$T_{\mathcal X}(M)$ of $M$ defined by 
\[ T_{\mathcal X}(M)  :=  \{x \in M \mid m.x = 0, x \in X \} \]
is a submodule of $M$ and is known as the $\mathcal X$-torsion submodule of $M$.  We will abuse notation somewhat and call it the $F \s C$-torsion submodule of $M$.

\section{The GK dimension of an $F \s A$-module}

\label{GK-dim-bsc}

In this section we shall describe a dimension for finitely generated $F \s A$-modules introduced in 
\cite{BG:2000} that coincides with the GK dimension. More precisely, given a finitely generated $F \s A$-module $M$ it is shown in \cite{BG:2000} that $\gk(M)$ is the supremum of the ranks of subgroups $B \le A$ so that $M$ is not $F \s B$-torsion. We shall use the notation
\[ B \bowtie M \] 
to denote the fact that $M$ is not $F \s B$-torsion.
In \cite{BG:2000} it was also shown that for each choice of (free) generators $A : \langle x_1, x_2, \cdots, x_n \rangle$  there is a subset $\langle x_{i_1}, x_{i_2}, \cdots, x_{i_k} \rangle$ of the generators so that \[  \langle x_{i_1}, x_{i_2}, \cdots, x_{i_k} \rangle \bowtie M \]
and $k = \gk(M)$.
In this case we may localize $M$ at the nonzero elements of $F \s \langle x_{i_1}, x_{i_2}, \cdots, x_{i_k} \rangle$ and the nonzero module of fractions thus obtained is necessarily finite dimensional as a vector space over the quotient division ring of 
$F \s \langle x_{i_1}, x_{i_2}, \cdots, x_{i_k} \rangle$.
In this situation the following holds.

\begin{lemma}[Lemma 4.1 of \cite{Gupta:2011}]
\label{AG_reslt}
Suppose that $F \s A$ has a
finitely generated module
$M$ and $A$ has a subgroup $C$ with $A/C$ torsion-free,
$\rk(C) = \gk(M)$,
and $F \s C$ commutative. Suppose moreover that $M$
is not $F \s C$-torsion.
Then $C$ has a virtual complement $E$ in $A$
such that $F \s E$ is commutative. 
Furthermore given $\mathbb Z$-bases 
$\{ x_1,\cdots, x_r \}$ and $\{ x_1,\cdots, x_r, 
x_{r + 1}, \cdots, x_n \}$ for 
$C$ and $A$ respectively there exist monomials $\mu_j$, where $j = 
 r + 1, \cdots n $, 
in $F \s C$, 
and an integer $s > 0$ such that the monomials $\mu_j \bar {x}_j^s$ commute in $F \s A$.
\end{lemma}

\begin{remark}
\label{fi}
In the case $A/C$ is not torsion-free $A$ has a subgroup $A_1$ of finite index so that $A_1/C$ is torsion-free.  Moreover $M$ is also a finitely generated $F \s A_1$-module with the same GK dimension.
We may thus apply the foregoing lemma which gives a virtual complement $E_1$ of $C$ in $A_1$ so that $F \s E_1$ is commutative. 
But then $E_1$ is also a virtual complement of $C$ in $A$. Hence the first part of Lemma \ref{AG_reslt} holds true even when $A/C$ is not torsion-free.

\end{remark}

In practice it is much more useful to work with the so-called critical modules.  

\begin{definition}
An $F \s A$-module is said to be \emph{critical}  if $M$ is nonzero and every proper quotient $M/N$ of $M$ satisfies
\[ \gk(M/N) < \gk(M), \ \ \ \ \ \ \ \ \ \ 0 < N< M. \]\end{definition}

In \cite{BG:2000} the following property was shown for the GK dimension of $F \s A$-modules.

\begin{proposition}[Lemma 2.2 of \cite{BG:2000}]
\label{exact_fn}
Let $M$ be an $F \s A$-module with a submodule $N$.
Then 
\[ \gk(M) = \max(\gk(N), \gk(M/N)) \]
\end{proposition}

The usefulness of the concept of a critical module owes itself to the following fact.

\begin{proposition}[Lemma 2.5 of \cite{BG:2000}]
Each nonzero $F \s A$-module contains a critical submodule. 
\end{proposition}

\section{Finite length modules}

We have the following result.

\begin{theorem} 
\label{MT1} 
Let $M$ be a nonzero finitely  
generated  $F \s A$-module where  
$F \s A$ has center $F$.  
Let $C < A$ be a subgroup having a subgroup $C_0$ of  
finite index such that $F \s C_0$ is commutative.  
If $M$ is finitely generated as an $F \s C$-module then,    
\begin{enumerate} 
\item[(i)] $\gk(M) = \rk(C)$,  
\item[(ii)] $M$ is $F \s C$-torsion-free,  
\item[(iii)] $M$ has finite length.  
\end{enumerate}  
\end{theorem}  
 
\begin{proof} (i) 
We shall denote the module 
$M$ regarded as 
$F \s C$-module as $M_C$. 
By hypothesis $M_C$ is finitely generated and 
so $M_{C_0}$ is also finitely generated where $M_{C_0}$ denotes $M_C$ viewed as $F \s C_0$-module.
By \cite[Lemma 2.7]{BG:2000}, \[ \gk(M) = \gk(M_C) = \gk(M_{C_0}) . \]
If $\gk(M) < \rk(C)$ we may pick a subgroup $E_0 < C_0$ with 
$\rk(E_0) < \rk(C_0)$ such that $E_0 \bowtie M_{C_0}$ and $\gk(M) = \rk(E_0)$ . 

By Lemmma \ref{AG_reslt} and 
Remark \ref{fi}, 
$E_0$ has 
a virtual complement  
$E_1$ in $A$ such that $F \s E_1$ 
is commutative.
Since $\rk(E_1) + \rk(C_0)$ exceeds $\rk(A)$, therefore $E_1 \cap C_0 > \langle 1 \rangle$.
Moreover as $E_1E_0$ has finite index in $A$, hence $E_1C_0$ has finite index in $A$.
But $\overline{E_1 \cap C_0}$ is central in  $F \s (E_1C_0)$ and hence by Proposition \ref{c_a_fis}, $F \s A$ has center larger than $F$.
This is contrary to the hypothesis in the theorem. So \[ \gk(M) \ge \rk(C). \] But the GK-dimension of an algebra $\mathcal A$ bounds the GK-dimensions of it's  modules (\cite{KL:2000}[Proposition 5.1(d)]) and so \[ \gk(M) \le \gk(F \s C) = \rk(C). \]

(ii) Suppose that the $F \s C$-torsion submodule $T$ of $M$ is nonzero.  
We recall that $T$ is an $F \s A$-submodule of $M$.    
Applying part (i) of the theorem just established to $T$  
we obtain $\gk(T) = \rk(C)$.  
 
From Section \ref{GK-dim-bsc} we know that  in this case $C \bowtie T$. However, this is contrary to the definition of $T$ as the $F \s C$-torsion submodule of $M$. 
Hence $T = 0$ and $M$ is $F \s C$-torsion-free. 

(iii) We first note that by part (i) and Proposition \ref{exact_fn} each nonzero subfactor of $M$ has the same GK dimension as $M$.  It now follows from  \cite[Lemma 5.6]{MP} and \cite[Section 5.9]{MP} 
that every descending chain in $M$ is eventually constant.

\end{proof}

\section{The GK dimensions of simple modules}

In this section we aim to determine the GK dimensions of simple $F \s A$-modules. 
This may depend to a large extent on the defining parameters $q_{ij}$. In \cite{MP} it was shown that when $\dim (F \s A) = 1$, then each simple $F \s A$-module has GK dimension equal to $\rk(A) - 1$. 

In general the GK dimensions of simple $F \s A$-modules depends on the defining co-cycle. But in the case $\dim(F \s A) = \rk(A) - 1$, we have the following result.

\begin{theorem}
Let $M$ be a simple-$F \s A$-module, where $\dim(F \s A) = \rk(A) - 1$. Let $Z$ be the subgroup of $A$ such that $\mathcal Z(F \s A) = F \s Z$ Then either $\gk(M) = 1$ or 
\[ \gk(M) \ge \rk(A) - \rk(Z) - 1.  \]    
\end{theorem}

\begin{remark}
We observe that $\rk(Z) \le \rk(A) - 2$. Otherwise $A$ contains a subgroup $A'$ of finite index such that $F \s A'$ is commutative but this would contradict Theorem \ref{Brookes_thm}.
\end{remark}

\begin{proof}
Since $\dim(F \s A) = \rk(A) - 1$, by Theorem \ref{Brookes_thm} $A$	 contains a subgroup $C$ with co-rank one such that $F \s C$ is commutative. However note that $A/C$ need not be infinite cyclic. In the case it is  
\begin{equation}
\label{sku_L_extn}
 F \s A = (F \s C)[X_n^{\pm 1}, \sigma],
\end{equation}
 
that is, $F \s A$ is a skew-Laurent extension of $F \s C$. Here $\sigma$ is the automorphism of $F \s C$ defined by $\gamma \mapsto X_n \gamma X_n^{-1}$, where $X_n$ denotes the image of a generator of $A$ modulo $C$.  

In the general case let $C$ be embedded in a subgroup $A'$ so $A'/C$ is infinite cyclic. Then the remarks made above apply to $F \s A'$. Now $F \s A$ is a finite normalizing extension of $F \s A'$ and so (e.g., \cite[Exercise 15A.3]{Rowen:2008}) 
a simple $F \s A$-module is a finite direct sum of simple $F \s A$-modules. 
As the GK dimension of a direct sum is the maximum of the of the GK dimensions of its summands (\cite[Proposition 5.1]{KL:2000}), it suffices to prove the theorem for $F \s A'$. We may thus assume that $A/C$ is infinite cyclic.

Now let $S$ be a simple $F \s A$-module. Let 
$L$ be a finitely generated critical $F \s C$-submodule of $S$. Consider the $F \s A$ submodule $L_1 : = L(F \s A)$. Since $S$ is simple $L_1 = S$. 

Now $L_1 = \sum_{i \in \mathbb Z} LX^i$. 
There is a surjective map $\theta$ from the induced module
\[ \oplus_{i \in \mathbb Z} LX^i \cong L \otimes_{F \s B} F \s A\]
to $L_1$. If this map is an isomorphism then $L$ is also simple and \[ \gk(L_1) = \gk(L) + 1 \] by 
\cite[Lemma 2.3]{BG:2000}.
Since $F \s C$ is commutative affine algebra $L$ is finite dimensional (over $F$) by the Nullstellensatz (see \cite[Section ]{MR:2001}). As the GK dimension of a finite dimensional module is zero, it follows that $\gk(L_1) = 1$. 

Now suppose that $\ker \theta \ne 0$. Then $\gk(L) = \gk(L_1)$ by \cite[Lemma 2.4]{BG:2000}.
Let $l := \gk(L_1)$ and $C = \langle y_1, y_2, \cdots, y_{n - 1} \rangle$. 
Since $A/C$ is infinite cyclic we may write $A = \langle y_1, y_2, \cdots, y_{n- 1}, y_n \rangle$.

By the criterion for the GK dimension of an $F \s A$-module (Section \ref{GK-dim-bsc}) we can pick $l$ generators, say 
$y_1, y_2, \cdots, y_l$ such that 
\[ \langle y_1, y_2, \cdots, y_l \rangle \bowtie L. \]
But then 
\[ \langle y_1, y_2, \cdots, y_l \rangle \bowtie L_1. \]
By Lemma \ref{AG_reslt}, 
there are monomials $\mu_1, \mu_2, \cdots, \mu_{n- l}$ in $\bar{y}_1, \cdots, \bar {y}_l$ so that 
\[ \mu_1\bar{y}_{l + 1}^s, \mu_2\bar{y}_{l +2}^s, \cdots, \mu_{n -l}\bar{y}_n^s \]
is a system of $n - l$ independent commuting monomials.
It is clear that the first $n - l - 1$ monomials centralize $F \s C$ since $F \s C$ is commutative.  

Next we observe that if $1 \le t \le n - l - 1$ then
\[ 1 = [\mu_t\bar{y}_{l + t}^s, \mu_{n -l}\bar{y}_n^s] =  [\mu_t\bar{y}_{l + t}^s, \bar{y}_n^s] = [(\mu_t\bar{y}_{l + t}^{s})^s, \bar{y}_n] = [\mu_t^s\bar{y}_{l + t}^{s^2}, \bar{y}_n] \]
noting Section \ref{com_cal}. 
Hence the system $\{\mu_t^s\bar{y}_{l + t}^{s^2}\}_{t  =  1}^{n - l - 1}$ of independent monomials centralizes the algebra $F \s A$.

It follows that  \[ \rk(Z) \ge n - l - 1 \]
which gives the desired inequality.
\end{proof}

\begin{theorem}
Suppose that $\dim(F \s A) = \rk(A) - 1$. There exists a simple $F \s A$-module $M$ with \[ \gk(M) \ge \rk(A) - \rk(Z) - 1, \] where $Z$ is the unique subgroup of $A$ so that 
\[ \mathcal Z(F \s A) = F \s Z. \] 

\end{theorem}

\begin{proof}
Set $\mathcal S = F \s C \setminus \{0\}$. Then using equation (\ref{sku_L_extn}), the localization $(F \s A)\mathcal S^{-1}$ is a skew-Laurent extension 
\[ R: = (F \s A)\mathcal S^{-1} = \mathscr D_C [X^{\pm 1}, \sigma] \] 
where $\mathscr D_C$ stands for the quotient division ring of $F \s C$ and $\sigma$ is the automorphism of $\mathscr D_C$ defined by 
$\sigma(f) = X_nfX_n^{-1}$ for $f \in \mathscr D_C$. Let $r$ be an irreducible element in the non-commutative PLID $R$ such that $(F \s A) \cap rR$ contains a unitary element of $F \s A$, that is, an element $\alpha$ such that in the unique expression 
$\alpha = \sum_{i = l}^h\gamma_iX_n^{i}$ the coefficients of the lowest and the highest powers of $x_n$ are units in $F \s C$. 

It is not difficult to find examples of such elements $r$. For
example, $u_1X_n + u_2$, where $u_i$ is a unit in $F \s C$ is a linear polynomial in $\mathscr D_C[X^{\pm 1}, \sigma]$ and so is irreducible.  Let \[ M(r): = (F \s A)/J(r), \]  where $J(r) := (F \s A) \cap rR$. Note that $rR$ is a maximal right ideal in $R$. By \cite[Lemma 3.4(1)]{BvO:2000}, we know that $M(r)$ is a simple $F \s A$-module if and only if \[ \Hom_{F \s A}(M(r), N) = 0 \]  for all $F \s C$-torsion simple $F \s A$-modules $N$. Suppose this last condition holds true so that $M(r)$ is a $F \s C$-torsion-free simple $F \s A$-module. Moreover as it is $F \s C$-torsion-free \[ \gk(M(r)) \ge \rk(C) = n - 1. \] 

Otherwise by \cite[Proposition 2.1]{Artamonov:1999}, $M(r)$ and so $N$ is finitely generated as $F \s C$-module. Now a reasoning similar to the one given in the proof of the last theorem with $N$ in place of $L$ shows that 
\[ \gk(N) \ge \rk(A) - \rk(Z) - 1. \]

\end{proof}


\begin{thebibliography}{99} 

\bibitem{Artamonov:1999} Artamonov, V. A. General quantum polynomials: irreducible modules and Morita equivalence. (Russian) Izv. Ross. Akad. Nauk Ser. Mat. 63 (1999), no. 5, 3--36

\bibitem{Brookes:2000}
Brookes, C. J. B., Crossed products and finitely presented groups. J. Group Theory 3 (2000), no. 4, 433–444.


\bibitem{BG:2000} C. J. B. Brookes, J. R. J. Groves, Modules over crossed products of a division ring with
an abelian group I, J. Algebra, 229, 2000, pp. 25-54.

\bibitem{BvO:2000} Bavula, V., van Oystaeyen, F., Simple holonomic modules over the second Weyl algebra A2. Adv. Math. 150 (2000), no. 1, 80–116.

\bibitem{Coutinho:1995} Coutinho, S. C., A primer of algebraic D-modules. London Mathematical Society Student Texts, 33. Cambridge University Press, Cambridge, 1995. xii+207 pp.



\bibitem{GW:1989} Goodearl, K. R., Warfield, R. B., Jr. An introduction to noncommutative Noetherian rings. London Mathematical Society Student Texts, 16. Cambridge University Press, Cambridge, 1989. 


\bibitem{Gupta:2011} Gupta, A., Modules over quantum Laurent polynomials., J. Aust. Math. Soc. 91 (2011), no. 3, 323–341.

\bibitem{Ja} V.~A.~Jategaonkar, `A multiplicative analogue of the Weyl algebra', \emph{Comm. Algebra} \textbf{12} (1984), 1669--1688.




\bibitem{KL:2000} G. R. Krause, T. H. Lenagan, Growth of Algebras and Gelfand-Kirillov Dimension,American Mathematical Soc. 2000.


\bibitem{Lo} M.~Lorenz,
\emph{Group Rings and Division Rings}, \emph{Methods in Ring Theory}, 265-280 (D. Reidel Publ. Co., Dordrecht, 1984).



\bibitem{MR:2001} McConnell, J. C., Robson, J. C., Noncommutative Noetherian rings. With the cooperation of L. W. Small. Revised edition. Graduate Studies in Mathematics, 30. American Mathematical Society, Providence, RI, 2001


\bibitem{MP} J.~C.~McConnell and J.~J.~Pettit, `Crossed products and multiplicative analogues of Weyl Algebras',
\emph{J. London Math. Soc.} \textbf{38} (1988), 47--55. 


\bibitem{Pass:1989} Passman, D. S., Infinite crossed products, Pure and Applied Mathematics, 135. Academic Press, Inc., Boston, MA, 1989. xii+468 pp.

\bibitem{Rob:1996} Derek~J.~S. Robinson, \emph{A Course in the Theory of Groups}, Springer-Verlag, New York, 1996. 
 

\bibitem{Rowen:2008} Rowen, L. H., Graduate algebra: noncommutative view. Graduate Studies in Mathematics, 91. American Mathematical Society, Providence, RI, 2008.


\end{thebibliography}
\end{document}